\documentclass[12pt]{amsart}
\usepackage{mathrsfs}
\usepackage{fancyhdr}
\usepackage{latexsym, amsfonts, amssymb, amsmath,  amsthm}
\usepackage{amsopn, amstext, amscd,pifont}
\usepackage{amssymb,bm, cite, color}
\usepackage[all]{xy}
\usepackage{hyperref}
\usepackage{thmtools}
\declaretheorem[name=Theorem,%
	postheadhook=\normalfont\slshape,%
	within=section]{theorem}
\declaretheoremstyle[headfont=\normalfont\bfseries,%
	sibling=theorem]{myplainstyle}
\declaretheorem[style=myplainstyle,name=Proposition]{proposition}
\declaretheorem[style=myplainstyle,name=Lemma]{lemma}
\declaretheorem[style=myplainstyle,name=Remark,]{remark}
\declaretheorem[style=myplainstyle,name=Definition]{definition}

\declaretheorem[style=myplainstyle,name=Fact]{fact}
\declaretheorem[style=myplainstyle,name=Fact-Notation]{fact-notation}
\declaretheorem[style=myplainstyle,name=Corollary]{corollary}
\declaretheorem[style=myplainstyle,name=Convention]{convention}

\newcommand{\Lang}{\mathscr{L}}
\newcommand{\Lr}{\mathscr{L}_\mathrm{r}}
\newcommand{\Lp}{\mathscr{L}_p}
\newcommand{\LR}{\mathscr{L}_\mathrm{R}}
\newcommand{\Ltheta}{\mathscr{L}_\theta}
\newcommand{\Ldiv}{\mathscr{L}_\mathrm{div}}
\newcommand{\Lpdiv}{\mathscr{L}_{p,\mathrm{div}}}
\newcommand{\LRdiv}{\mathscr{L}_{R,\mathrm{div}}}
\newcommand{\Lthetadiv}{\mathscr{L}_{\theta,\mathrm{div}}}
\newcommand{\LRpdiv}{\mathscr{L}_{R, p,\mathrm{div}}}
\newcommand{\Lthetapdiv}{\mathscr{L}_{\theta, p,\mathrm{div}}}
\newcommand{\M}{\mathcal{M}}
\newcommand{\N}{\mathcal{N}}

\newcommand{\alg}{\mathrm{alg}}
\newcommand{\sep}{\mathrm{sep}}

\renewcommand{\Im}{\mathrm{Im}}

\newcommand{\PACVF}{\mathrm{PACVF}}
\newcommand{\ACVF}{\mathrm{ACVF}}
\newcommand{\FrobVF}{\mathrm{FrobVF}}

\newcommand{\SCVF}{\mathrm{SCVF}}

\newcommand{\Frac}{\mathrm{Frac}}
\newcommand{\Gal}{\mathrm{Gal}}

\newcommand{\Abs}{\mathrm{Abs}}

\numberwithin{equation}{section}

\begin{document}
\title[QE on some PACVFs]{Quantifier elimination
on some pseudo-algebraically closed valued fields}
\author{Jizhan Hong}
\address{School of Mathematical Sciences \\ Huaqiao University \\ 269 Chenghua
	North\\ Quanzhou
Fujian \\ China, 362021}
\email{shuxuef@hqu.edu.cn}

\begin{abstract}
	Adjoining to the language of rings the function symbols for splitting 
	coefficients, the function symbols for relative $p$-coordinate functions,
	and the division predicate for a valuation, 
	some theories of 
	pseudo-algebraically closed non-trivially
	valued fields admit quantifier elimination. 
	It is also 
	shown that in the same language the theory of pseudo-algebraically closed 
	non-trivially valued fields of a given exponent of imperfection does 
	not admit quantifier elimination, due to Galois theoretic obstructions.
\end{abstract}

\maketitle 

\tableofcontents

\section{Introduction}

The study of pseudo-algebraically closed fields originates from Ax's study on 
the elementary theory of finite fields, \cite{MR0229613}, in the 1960s. A class 
of fields were isolated by Ax, called pseudo-finite fields. It was proved that a 
pseudo-finite field $K$ has the property that every absolutely irreducible 
affine algebraic set defined over $K$ has a $K$-rational point. This property 
was later used to define pseudo-algebraically closed fields. Many known results 
on pseudo-algebraically closed fields
are documented in the book \cite{MR2445111}.  Examples of pseudo-algebraically
closed fields include separably closed fields, pseudo-finite fields
and some Hilbertian fields. Global fields and non-trivial 
purely transcendental extensions
over any field are Hilbertian fields.  
It was proved in \cite{MR0302651} that if $K$ is a countable Hilbertian field, 
then with respect to the Haar measure, 
for almost all tuples $(\sigma_1, \ldots, \sigma_e)\in \Gal(K)^e$, the fixed 
field of the tuple is a pseudo-algebraically closed field. A lot is known on the 
model theory of pseudo-algebraically closed fields. A characterization of 
the elementary equivalence relation among
pseudo-algebraically closed fields in terms of Galois groups was given 
in \cite{CherlinvddMacintyre:ThRCfields}. The model theory of pseudo-algebraically
closed fields turns out to depend very much on the Galois groups. Using an important
Embedding Lemma 
\cite{MR0435050}, it was possible to prove that certain class of pseudo-algebraically
closed fields admits quantifier elimination in a suitably nice language. 
While the theory of pseudo-algebraically closed fields itself in the same language
does not admit quantifier elimination due to some Galois theoretic obstructions, 
it was however shown that the theory itself admits a new kind of
`quantifier elimination' using the so-called Galois stratifications. 
Upon a close study of the so-called `comodel theory' over these 
pseudo-algebraically closed fields, in \cite{CherlinvddMacintyre:ThRCfields},
it was also proved that the theory of perfect
Frobenius fields admits quantifier elimination
 in a so-called `Galois formalism'.

In this article, we are interested in the model theory of pseudo-algebraically 
closed fields endowed with a non-trivial valuation. It was proved in 
\cite{MR2350157} that if $K$ is pseudo-algebraically closed, $V$ is 
any valuation on the algebraic closure $K^\alg$ of $K$,  and $X$ is 
a geometrically integral $K$-variety, then $X(K)$ is dense in $X(K^\alg)$
with respect to the topology induced by $V$. This provides a critical tool 
in the study of pseudo-algebraically closed valued fields. It enables 
us to prove the Valuation Theoretic Embedding Lemma and the 
quantifier elimination results in this article.

The strategy for proving the quantifier elimination results in this article
turns out to be almost similar to the field theoretic case; one only needs to 
take a little care of the valuations.  The main result in this article is 
in Section \ref{Section-quantifier-elimination}. Since that section is not 
very long, we will not give another summary here. But roughly speaking, it says that 
the theory of non-trivially valued Frobenius fields with a fixed 
exponent of imperfection has quantifier elimination 
in a sufficiently nice language. Similar to the field theoretic case, the same 
Galois theoretic obstructions render the quantifier elimination for the 
theory of pseudo-algebraically closed valued fields with a fixed 
exponent of imperfection impossible (see Section \ref{section-lack-of-QE}). 
It is also possible to obtain a similar characterization
(Section \ref{elementary-invariants}) of the 
elementary equivalence relation of pseudo-algebraically closed valued fields
in terms of the `comodel theory' as in \cite{CherlinvddMacintyre:ThRCfields}.

The motivation of the author to study pseudo-algebraically closed valued fields
is an attempt to understand how far the first order valued-field structure of a 
pseudo-algebraically closed valued field is from the first order field structure. 
However, this topic will be postponed for another occasion.

\section{Preliminaries}

In this section, we recall some definitions and facts that will 
be will in this article. All the results in this sections are known. 
This section is only meant to serve as a clarification of related 
terminology, not as a thorough review of the related definitions or results.

\subsection{Field extensions}

Given a field $K$, we use $K^\alg$ to denote the algebraic closure of $K$, 
$K^\sep$ the separable closure of $K$. Unless otherwise mentioned, the characteristic 
of a field $K$ is denoted by the letter $p$, which could be either $0$ or a positive 
prime integer.
For an integral domain $A$, we use $A^\times$ to denote the 
set of all invertible elements in $A$. If $K$ is a field, the absolute 
Galois group of $K$ is denoted by $\Gal(K):=\Gal(K^\sep/K)$.

\begin{definition}
	Suppose that two subfields of $L$, denoted by $E$ and $F$, contain 
	a common subfield $K$. Then $E$ is said to be {\bf linearly disjoint}
	over $K$ if any finitely many 
	linearly independent elements $x_1, \ldots, x_n\in E$ over $K$ 
	remains linearly independent over $F$. 
	we say that $E$ is {\bf algebraically independent} from $F$ over $K$, 
	or that $E$ and $F$ are {\bf algebraically independent} over $K$, 
	if any finitely many 
	algebraically independent elements $x_1, \ldots, x_n\in E$ over $K$ 
	remains algebraically independent over $F$. 
\end{definition}
\begin{definition}
	A field extension $L/K$ is said to be {\bf separable} if 
	$L$ is linearly disjoint from $K^{1/p^{\infty}}$. 
	A field extension $L/K$ is said to be {\bf regular} if 
	$L$ is linearly disjoint from $K^\alg$.  It is well-known that
	$L/K$ is regular if and only if $L/K$ is separable and $K$ is 
	relatively algebraically closed in $L$.
\end{definition}
\begin{definition}
	Suppose that $K$ is a field of positive characteristic $p$. Denote by 
	$K^p$ the subfield of $p$-th powers of $K$. Then $[K:K^p]$ is either 
	infinite or $p^n$ for some $n\in \omega$, called the 
	{\bf degree of imperfection} of $K$.  The {\bf exponent of imperfection} 
	of $K$ is defined to be $n$, if $[K:K^p]=n$, or $\infty$, if $[K:K^p]=
	\infty$. Given finitely many elements $x_1, \ldots, x_n\in K$, 
	an element of the form 
	$$ \prod_{j=1}^n x_j^{i(j)}, \quad i:n\to p,$$
	is called of {\bf $p$-monomial} in $x_1, \ldots, x_n$; 
	$x_1, \ldots, x_n$ are said to be {\bf $p$-independent} over a 
	subfield $k$ if the set of all $p$-monomials are linearly 
	independent over $k$. If $x_1, \ldots, x_n$ 
	are $p$-independent over $K^p$, then we also say that $x_1, \ldots, x_n$
	are $p$-independent in $K$.  A subset of $K$ is called a {\bf $p$-basis}
	if it is maximally $p$-independent in $K$.
\end{definition}

It is well-known that $[K:K^p]=p^n$ if and only if $K$ has a $p$-basis of 
cardinality $n$. Furthermore, if $\{x_1, \ldots, x_n\}$ is a $p$-basis of 
$K$, then $K=K^p(x_1, \ldots, x_n)$; every element $y\in K$ can be uniquely 
written as a linearly combination of $p$-monomials in $x_1, \ldots, x_n$
over $K^p$.

\begin{definition}
	If $K$ is a field of characteristic $0$, then $K$ is perfect. 
	the degree of imperfection of $K$ is defined to be $1$ and the 
	exponent of imperfection of $K$ is defined to be $0$.
\end{definition}

It is well-known that a field extension $L/K$ is separable if and only if 
every $p$-independent set of $K$ remains $p$-independent in $L$, 
if and only if there a $p$-basis of $K$ remains $p$-independent in $L$.

\begin{fact}
	[Lemma 2.7.3 of \cite{MR2445111}]
	\label{separable-algebraic-extension-imperfection-degree}
	Separable algebraic extensions preserve degrees of imperfection. 
\end{fact}

\begin{definition}
	[see \cite{MR853851}]
	\label{definition-relative-p-coordinate-funtions}
	Suppose that $K$ is a field of positive characteristic $p$. 
	We define the {\bf relative $p$-coordinate functions}  
	$\lambda_{n,i}(x;y_1,\ldots, y_n)$ as follows: if $y_1, \ldots, y_n$
	are not $p$-independent in $K$ or $x\not\in K^p(y_1, \ldots,y_n)$, 
	then $\lambda_{n,i}(x;y_1, \ldots, y_n)=0$, otherwise 
	$\lambda_{n,i}(x; y_1, \ldots, y_n)$ is the unique $i$-th coordinate of 
	$x$ with respect $y_1, \ldots, y_n$ when $x$ is written 
	as a linearly combination of $p$-monomials in $y_1, \ldots, y_n$
	over $K^p$.
\end{definition}

\subsection{Valued fields}
\label{subsection-valued-fields}

\begin{definition}
	Given a field $K$, a subring $V$ is called a {\bf valuation} or 
	a {\bf valuation ring} on $K$ if for all $x\in K^\times$, either 
	$x\in V$ or $x^{-1}\in V$. A {\bf valued field} is a field $K$ with 
	a distinguished valuation ring $V$, usually denoted by the pair 
	$(K,V)$.
	For a valued field $(K, V)$, 
	we use the corresponding lower case letter $v$ to denote the valuation 
	map associated to $V$, 
	we also use $vK^\times$ to denote the value group of $(K,V)$ and use
	$Kv$ to denote the residue field of $(K,V)$.
\end{definition}
\begin{fact}
	[Theorem 3.1.2 of \cite{EnglerPrestel:ValuedFields}]
	\label{extension-of-valuation-exists}
	Let $L/K$ be a field extension, $V$ a valuation ring on $K$. Then 
	there exists a valuation ring $W$ on $L$ extending $K$, i.e.~$W\cap K=V$.
\end{fact}
\begin{fact}
	[Lemma 3.2.8 of \cite{EnglerPrestel:ValuedFields}]
	\label{two-extensions-of-valuation-equal}
	Suppose that $L/K$ is an algebraic extension, $V$ a valuation ring on $K$. 
	Suppose that $W_1$ and $W_2$ are two extensions of $V$ to $L$. 
	If $W_1\subseteq W_2$ then $W_1=W_2$.
\end{fact}
\begin{fact}[Theorem 3.2.15 of \cite{EnglerPrestel:ValuedFields}]
	\label{conjugate-theorem-normal-extensions}
	Suppose that $(L, W_1)$ and $(L, W_2)$ are
	valued-field extensions of a valued field $(K,V)$ where $L/K$
	is a normal field extension. Then there exists $\sigma \in \mathrm{Aut}(L/K)$
	such that $\sigma(W_1)=W_2$.
\end{fact}
Give a valued field extension $(L,W)$ of $(K,V)$, the residue field $Kv$ naturally
embeds into $Lw$ and the value group $vK^\times$ naturally embeds into $wL^\times$.
Denote $[Lw:Kv]$ by $f(W/V)$, called the {\bf residue degree} of $W/V$.; denote 
$(wL^\times:vK^\times)$ by $e(W/V)$, called the {\bf ramification index} of $W/V$.
\begin{fact}
	[Lemma 3.3.2 of \cite{EnglerPrestel:ValuedFields}]
	\label{Galois-extension-of-valuations}
	Suppose that $L/K$ is a Galois extension of degree $n$. Suppose that
	$V$ is a valuation on $K$ and $V_1, \ldots, V_r$ are all the extensions of 
	$V$ to $L$. All the ramification indexes $e(V_i/V)$ are equal to a number $e$
	and 
	all the residue degrees $f(V_i/V)$ are also equal to a number $f$. 
	Furthermore, we have 
	$ n=ref.$
\end{fact}
\begin{fact}
	[Theorem 3.2.4 of \cite{EnglerPrestel:ValuedFields}]
	\label{algebraic-extension-of-valuations}
	Suppose that $(L,W)$ is a valued field extension of $(K,V)$ with $L/K$
	algebraic. Then $Lw$ is an algebraic extension of $Kv$ and 
	$wL^\times$ is contained in the divisible hull of $vK^\times$.
\end{fact}
\begin{fact}
	[Theorem 3.2.11 of \cite{EnglerPrestel:ValuedFields}]
	\label{residue-field-and-value-group-of-SCVF}
	Suppose that $K$ is separably closed and non-trivially valued, then 
	its residue field is algebraically closed and its value group is divisible.
\end{fact}
\subsection{Profinite groups}

Given a field $K$, the absolute Galois group $\Gal(K):=\Gal(K^\sep/K)$ of $K$ 
is naturally a profinite group. There is also a version of Galois correspondence 
for the Galois extension $K^\sep/K$.
For a positive integer, we use $\hat{F}_e$ to denote the {\bf free profinite group}
of rank $e$, and $\hat{F}_\omega$ to denote the free profinite 
group of rank $\aleph_0$.
For a profinite group $G$, we use $\Im(G)$ to denote the set of all 
ismorphism types of finite quotients of $G$.

\begin{convention}
	Unless otherwise mentioned, 
	morphisms between profinite finite groups in this article are always
	assumed to the continuous.
\end{convention}

\begin{definition}
	[see \cite{MR2599132}]
	Let $G_i$, $i=1, \ldots, n$, be a finite collection of profinite groups. 
	A {\bf free profinite product} of these groups  consists of a profinite 
	group $G$ and continuous homomorphisms $\varphi_i: G_i\to G$, 
	$i=1, \ldots, n$, satisfying the following universal property:
	$$
	\xymatrix{
		G \ar@{.>}[dr]^{\psi}&  \\
		G_i \ar[u]^{\varphi_i} \ar[r]_{\psi_i} & K	
		}
	$$
	for any profinite group $K$ and any continuous homomorphisms
	$\psi_i : G_i \to K$,  $i=1, \ldots, n$, there is a unique 
	continuous homomorphism $\psi: G\to K$ such that $\phi_i=\psi \varphi_i$
	for all $i=1, \ldots, n$. The free profinite product thus defined always 
	exists and is unique up to isomorphism, usually denoted by 
	$$ G = \coprod_{i=1}^n G_i.$$
\end{definition}
\begin{fact}
	[Corollary 9.1.4 of \cite{MR2599132}]
	Let $G_1, \ldots, G_n$ be profinite groups and $G=G_1\coprod \cdots \coprod
	G_n$ be their free profinite product. Then the following conclusions hold:
	\begin{enumerate}
		\item the natural homomorphisms $\varphi_j: G_j\to \coprod_{i=1}^n
			G_i$, $j=1, \ldots, n$ are monomorphisms; 
		\item $G=\overline{
				\left<\varphi_i(G_i)\mid i=1, \ldots, n\right>
				}$.
	\end{enumerate}
\end{fact}

\begin{fact}
	[Theorem 9.1.12 of \cite{MR2599132}]
	\label{selfnormalizing-component-of-free-product}
	Let $G_1, \ldots, G_n$ be profinite groups and let $G=G_1\coprod
	\cdots \coprod G_n$ be their free profinite product. Then $N_G(G_i)=G$ for 
	all $i=1,\ldots, n$.
\end{fact}
\begin{fact}[Prop.~17.6.2 in \cite{MR2445111}]
	\label{open-subgroups-of-free-profinite-groups}
	Let $\hat{F}$
	be a free profinite group, and $\hat{H}$ an open subgroup of $\hat{F}$.
	Suppose that $\hat{H}$ is profinite. Then $\hat{H}$ is a
	free profinite group. Moreover, if the rank $e$ of $\hat{F}$ is 
	finite, then the rank of $\hat{H}$ is $1+(\hat{F}:\hat{H})(e-1)$; 
	if the rank of $\hat{F}$ is infinite, then $\hat{H}$ and $\hat{F}$ have 
	the same rank.
\end{fact}
\begin{definition}
	A profinite group $G$ is {\bf projective} if 
	for every diagram 
	$$
	\xymatrix{
		& G \ar[d]^{\varphi} \\
		B \ar[r]_{\alpha} & A 
		}
	$$ 
	where $\varphi$ and $\alpha$ are epimorphisms of profinite groups, 
	there exists a homomorphism $\gamma: G\to B$ such that the following 
	diagram commutes
	$$
	\xymatrix{
		& G \ar[d]^{\varphi} \ar[dl]_{\gamma} \\
		B \ar[r]_{\alpha} & A 
		}
	$$ 
\end{definition}
\begin{fact}
	[Prop.~22.4.10 of \cite{MR2445111}]
	\label{free-product-of-projective-groups}
	The free profinite product two projective profinite groups is again
	projective.
\end{fact}
\begin{fact}
	[Gasch\"utz, Lemma 17.7.2 of \cite{MR2445111}]
	\label{Gaschutzs-lemma}
	Let $\pi: G\to H$ be an epimorphism of profinite groups with 
	the rank of $G$ not more than a positive integer $e$. Let $h_1, \ldots,
	h_e$ be a system of generators of $H$. Then there exists a system of
	generators $g_1, \ldots, g_e$ of $G$ such that $\pi(g_i)=h_i$ for all 
	$i=1, \ldots, e$.
\end{fact}
\begin{corollary}
	[Proposition 17.7.3 of \cite{MR2445111}]
	\label{corollary-of-Gaschutzs-lemma}
	Let $e$ be a positive integer, $\varphi : \hat{F}_e \to H$ and $\alpha:
	G\to H$ a pair of epimorphisms of profinite groups where
	the rank of $G$ is not more than $e$. Then there exists an epimorphism
	$\gamma : \hat{F}_e\to G$ such that the following diagram commutes.
	$$
	\xymatrix{
		 & \hat{F}_e \ar[dl]_{\gamma} \ar[d]^{\varphi} \\
		G \ar[r]_{\alpha} & H 
		}
	$$

\end{corollary}
\subsection{Pseudo-algebraically closed fields and Hilbertian fields}
\begin{definition}
	A field $K$ is {\bf pseudo-algebraically closed} or {\bf PAC} if 
	every (non-empty)
	absolutely irreducible affine algebraic set defined over $K$ has a
	$K$-rational point. In scheme theoretic terms: $K$ is PAC if every 
	geometrically integral $K$-variety has a $K$-point.

	For a positive integer $e$, we say $K$ is {\bf $e$-free} if the absolute 
	Galois group $\Gal(K)$ of $K$ 
	is isomorphic to the free-profinite group 
	$\hat{F}_e$ of rank $e$.

	A field $K$ is {\bf pseudo-finite}, if $K$ is PAC, perfect and $1$-free.
\end{definition}
\begin{fact}
	[Corollary 11.2.5 of \cite{MR2445111}]
	\label{algebraic-extensions-of-PAC}
	Every algebraic extension of a PAC field is PAC.
\end{fact}
\begin{fact}
	[Corollary 11.5.5 of \cite{MR2445111}]
	\label{PAC-residue-field-and-value-group}
	Suppose that $K$ is a PAC field with a non-trivial valuation. 
	Then $K^\sep$ is an immediate extension of $K$. Thus by Fact
	\ref{residue-field-and-value-group-of-SCVF}, 
	the residue field of $K$ is algebraically closed and the value group 
	$K$ is divisible. 
\end{fact}

The following is called {\it the PAC Nullstellensatz}  op. cit.
\begin{fact}
	[Theorem 18.6.1 of \cite{MR2445111}]
	\label{PAC-Nullstellensatz}
	Let $K$ be a countable Hilbertian field and $e$ a positive integer. 
	Then the fixed field of a tuple $\pmb{\sigma}\in \Gal(K)^e$, 
	$K^\sep(\pmb{\sigma})$, is a 
	PAC field for almost all $\pmb{\sigma} \in \Gal(K)^e$, that is, 
	there exists a subset $A\subseteq \Gal(K)^e$ with Haar measure $1$ 
	such that for all $\pmb{\sigma}\in A$, $K^\sep(\pmb{\sigma})$ is 
	PAC.
\end{fact}
\begin{fact}
	[Theorem 18.5.6 of \cite{MR2445111}]
	\label{so-called-The-Free-Generator-Theorem}
	Let $K$ be a Hilbertian field and $e$ a positive integer. Then
	$\left<\sigma_1,\ldots, \sigma_e\right>\cong \hat{F}_e$ 
	for almost all $(\sigma_1, \ldots, \sigma_e)\in \Gal(K)^e$. 
\end{fact}

\begin{fact}
	[Theorem 18.10.2 of \cite{MR2445111}]	
	\label{density-theorem-for-smallest-Galois-closure-PAC}
	Let $K$ be a countable Hilbertian field and $e$ a positive integer. Then
	the maximal Galois extension of $K$ in $K^\sep(\pmb{\sigma})$, 
	denoted by $K^\sep[\pmb{\sigma}]$,
	is PAC for almost all $\pmb{\sigma}\in \Gal(K)^e$.
\end{fact}
\begin{fact}
	[Theorem 27.4.8 of \cite{MR2445111}]
	\label{density-theorem-for-smallest-Galois-closure-PAC-and-omega-free}
	Let $K$ be a countable Hilbertian field. Then, for almost all 
	$\pmb{\sigma}\in \Gal(K)^e$, the field $K^\sep[\pmb{\sigma}]$
	is $\omega$-free and PAC. 
\end{fact}
\begin{fact}
	[Theorem 2 of \cite{MR2350157}]
	\label{Kollar-density}
	Suppose that $K$ is a PAC field, $V$ a non-trivial valuation on $K^\alg$. 
	Then for any absolutely irreducible
	affine algebraic set $V$ defined over $K$, $V(K)$ is 
	$V$-dense in $V(K^\alg)$.
\end{fact}

\subsection{Languages and theories}
\begin{definition}
	We use the symbol $\Lr$ to denote the language of rings $\{+, - , \times,
	0, 1\}$,
	the symbol $\Ldiv$ to denote the (one-sorted)
	language of valued fields (see Subsection 
	\ref{subsection-valued-fields} for valued fields),
	namely the set $\Lr\cup\{\mid\}$, where the 
	vertical bar `$\mid$' is meant to be interpreted as the division predicate
	associated to the valuation, that is, 
	`$x \mid y$' if and only if the valuation of $x$ is not more than the 
	valuation of $y$. However, we usually use write $v(x)\leq v(y)$ 
	or $v(y)\geq v(x)$
	instead of $x \mid y$.
	We use the symbol $\ACVF$ to denote the theory of algebraically closed 
	non-trivially valued fields, in the language of $\Ldiv$.
\end{definition}
\begin{definition}
	If $p$ is a prime number, 
	we use $\Lp$ to denote the language $\Lr\cup\{\lambda_{n,i}(x;y_1, \ldots, 
	y_n)\}_{n\in \omega, i\in p^n}$, where the 
	function symbols are meant to be interpreted as the relative $p$-coordinate
	functions as defined in Definition 
	\ref{definition-relative-p-coordinate-funtions}.
	If $p$ is $0$, we use $\Lp$ to denote $\Lr$.
\end{definition}
\begin{definition}
	$\Lpdiv$ stands for $\Lp\cup \Ldiv$. 
\end{definition}
\begin{fact}[Prop.~4.3.28 of \cite{MR1924282}]\label{use-embedding-to-get-QE}
	If $\Lang$ is a language containing a constant symbol and $T$ 
	a theory in $\Lang$, then the following 
	are equivalent:
	\begin{enumerate}
		\item $T$ has quantifier elimination;
		\item for every $\M\models T$, $A\subset M$, 
			$\N\models T$ being $|M|^+$-saturated,
	and $f:A\to N$ a partial $\Lang$-embedding, $f$ extends to an 
	$\Lang$-embedding of $\M$
	into $\N$.
	\end{enumerate}
\end{fact}
\begin{fact}
	[Theorem 13.1 of \cite{MR0398817}]
	\label{substructure-complete} A theory $T$ has quantifier elimination 
	if and only $T$ is substructure complete, i.e.~for any substructure 
	$\mathcal{A}$ of a model $\M\models T$, the 
	new theory obtained by adjoining to $T$ the atomic diagram of $\mathcal{A}$,
	$T\cup\mathrm{Diag}(A)$, is complete.
\end{fact}
\begin{fact}\label{ACVF-QE}
	The theory $\ACVF$ admits quantifier elimination in $\Ldiv$.  Robinson 
	proved that $\ACVF$ is model-complete in \cite{MR0472504}. For a complete
	proof of the quantifier elimination result, consult for example 
	\cite{vandenDries:Valued}.
\end{fact}

\begin{remark}
	More languages will be introduced in Section \ref{section-lack-of-QE}.
\end{remark}

\section{Some lemmas}

\begin{lemma}
	\label{number-of-extensions-equals-degree}
	Suppose that $(K,V)$ is a valued field where $K$ is PAC and $V$ is 
	non-trivial. Then for any finite separable algebraic extension $L/K$, there 
	are exactly $[L:K]$ distinct extensions of $V$ to $L$.
\end{lemma}
\begin{proof}
	By Fact \ref{PAC-residue-field-and-value-group}, $Kv$ is algebraically 
	closed and $vK^\times$ is divisible. By Fact
	\ref{algebraic-extension-of-valuations}, for any 
	 extension $W$ of $V$ to $L/K$, the ramification index $e(W/V)$ and 
	 the residue degree $f(W/V)$ are both always equal to $1$. Let $\tilde{L}$
	 be the Galois closure of $L/W$, then by Fact
	 \ref{Galois-extension-of-valuations}, there are exactly $[\tilde{L}:K]$ 
	 extensions of $V$ to $\tilde{L}$.  By Fact \ref{algebraic-extensions-of-PAC},
	 $L$ is also PAC, since $\tilde{L}$ is also Galois over $L$, there are 
	 exactly $[\tilde{L}:L]$ extension of a given non-trivial valuation on $L$
	 to $\tilde{L}$. This means that there are exactly 
	 $$ \frac{[\tilde{L}:K]}{[\tilde{L}:L]} = [L:K]$$
	 extensions of $V$ to $L$.
\end{proof}
\begin{lemma}
	Suppose that $(K,V)$ is a valued field where $K$ is PAC and $V$ is 
	non-trivial. Suppose that $L/K$ is a
	finite separable algebraic extension. 
	If $L/K$ is not Galois, then
	there exist two valuations $V_1$ and $V_2$
	on $L$
	extending $V$ such that for any $\sigma \in \mathrm{Aut}(L/K)$, $\sigma(V_1)
	\neq V_2$.
\end{lemma}
\begin{proof}
	By Lemma \ref{number-of-extensions-equals-degree}, 
	there are $[L:K]$ distinct extensions of $V$ to $L$, but since 
	$L/K$ is not Galois, the number of elements in $\mathrm{Aut}(L/K)$ 
	is strictly less than $[L:K]$. Thus there are two valuations $V_1$ and 
	$V_2$ on $L$ extending $V$ such that they are not conjugated over $K$. 
\end{proof}
\begin{lemma}
	\label{distinguished-primitive-element-wrt-valuation}
	Suppose that $(K,V)$ is a valued field where $K$ is PAC and $V$ is
	non-trivial. Suppose that $L/K$ is a finite separable algebraic extension
	and that $V_1, \ldots, V_n$ are all the distinct extensions of $V$ to $L$. 
	Then there exists a primitive element $a$ for the field extension $L/K$ 
	such that $a\in V_1\backslash \left(\cup_{i=2}^n V_i\right)$.
\end{lemma}
\begin{proof}
	By \ref{PAC-residue-field-and-value-group}, the residue field $Kv$ is 
	infinite. Let
	$\Lambda:=\{\lambda_i\}_{i\in I}\subseteq V$ be a set of representatives
	of elements in $Kv$. Then $\Lambda \subseteq \cap_{i=1}^n V_i$.

	If $n=1$, then the conclusion is trivially true.  Thus we may assume that 
	$n\geq 2$. 

	First, we show that $V_1\backslash \left(\cup_{i=2}^n V_i\right)$ is not
	empty.  
	If $n=2$, then by the assumption that $V_1\neq V_2$, it follows from
	Fact \ref{two-extensions-of-valuation-equal} that $V_1\not\subseteq V_2$. 
	Thus $V_1\backslash V_2$ is not empty.
	If $n\geq 3$, then 
	suppose toward a contradiction that $V_1\subseteq \cup_{i=2}^n V_i$. Without
	loss of generality, we  may assume that the minimal number of $V_i$ 
	whose union contains $V_1$ is $k$ and $V_1\subseteq \cup_{i=2}^k V_i$. 
	Since $V_1, \ldots, V_n$ are distinct, $k\geq 3$.
	This means that there exists $s\in V_2\cap V_1$ and $t\in V_3\cap V_1$
	such that 
	$s \not\in \cup_{i=3}^k V_3$ and $t\not\in V_2\cup\left(\cup_{i=4}^k V_i
	\right)$. It follows that for all $\lambda \in \Lambda$, 
	$s+\lambda (t-s) \in V_1 
	\subseteq \cup_{i=2}^k V_i$. Note that $\Lambda$ is infinite.
	Therefore there exists $\lambda_1\neq \lambda_2$ and $2\leq i\leq k$
	such that $\lambda_1, 
	\lambda_2\in \Lambda$
	and both $p:=s+\lambda_1 (t-s)$ and $q:=s+\lambda_2 (t-s)$ are in 
	a common
	$V_i$. Then for $\mu = \frac{-\lambda_1}{\lambda_2-\lambda_1}\in V_i$, 
	$s=p+\mu(q-p)\in V_i$, and for 
	$\mu = \frac{1-\lambda_1}{\lambda_2-\lambda_1}\in V_i$, 
	$t=p+\mu(q-p)\in V_i$. Thus both $s$ and $t$
	are in $V_i$, contradicting the choices of $s$ and $t$.  Thus 
	$V_1\backslash \left(\cup_{i=2}^n V_i\right)$ is not empty.

	Now choose an element $a \in V_1\backslash\left(\cup_{i=2}^n V_i\right)$.
	Then any extension of the valuation ring $K(a)\cap V_1$ on $K(a)$
	to $L$ has to be the same as $V_1$. By
	\ref{number-of-extensions-equals-degree}, this means that $[L:K(a)]=1$,
	namely $L=K(a)$. Thus $a$ is a primitive element of $L/K$. 
\end{proof}
\begin{lemma}
	\label{distinguishing-nonconjugated-valuations}
	Suppose that $(K,V)$ is a valued field where $K$ is PAC, $V$ is 
	non-trivial, and $L/K$ is a finite separable algebraic extension. 
	Suppose that $V_1$ and $V_2$ are two extensions of $V$
	such that for all $\sigma \in \mathrm{Aut}(L/K)$, $\sigma(V_1)\neq V_2$.
	Then there exists a primitive element $a$ for the extension $L/K$ with 
	minimal polynomial $f(X)\in K[X]$ satisfying the following properties:
	\begin{enumerate}
		\item $a \in V_1 \backslash V_2$;
		\item for all $r\in L$ such that $f(r)=0$, $r\not\in V_2$.
	\end{enumerate}
\end{lemma}
\begin{proof}
	Denote by $V_3, \ldots, V_n$ the rest of the all the distinct extensions
	of $V$ to $L$, where $n=[L:K]$ by Lemma
	\ref{number-of-extensions-equals-degree}.
	Then by Lemma \ref{distinguished-primitive-element-wrt-valuation}, 
	there exists a primitive element $a$ for the extension $L$ such that 
	$a \in V_1 \backslash \left(\cup_{i=2}^n V_i\right)$. Denote the 
	minimal polynomial of $a$ over $K$ by $f(X)\in K[X]$, the Galois 
	closure of $L/K$ by $\tilde{L}$, the Galois group of $\tilde{L}/K$
	by $G$, the Galois group of $\tilde{L}/L$ by $H$, and the order of $H$ by 
	$d$. Then the $n$ roots of $f(X)$ over $K$ are all elements of $\tilde{L}$. 
	By Fact \ref{conjugate-theorem-normal-extensions} and by Lemma 
	\ref{number-of-extensions-equals-degree}, for any valuation $W$ on $\tilde{L}$
	extending $V$ and any $\sigma \in G$, if $\sigma \neq 1$, then $\sigma(W)
	\neq W$. For each $i=1, \ldots, n$, let the 
	valuation rings $W_{i1}, \ldots, W_{id}$ be all of 
	the distinct extensions of $V_i$ on $L$ to $\tilde{L}$.  
	
	It then follows that for all $\sigma \in G$,  $\sigma \in H$ 
	if and only if $\sigma(W_{11})\cap L= V_1$, if and only if $\sigma(W_{11})
	\in\{W_{11}, \ldots, W_{1d}\}$. Indeed, if $\sigma \in H$, then $\sigma$ 
	fixes $L$ element-wise. So $V_1= \sigma(V_1) = \sigma(W_{11}\cap L) =
	\sigma(W_{11})\cap L$. Meanwhile, there are exactly $d$ valuations on
	$\tilde{L}$ extending $V_1$, so there are exactly $d$ elements in $G$ 
	each of whose action on $W_{11}$ extends $V_1$. We now have found all 
	the $d$ elements of $H$ satisfies the later property, so these must be 
	all such elements.  

	Similarly, for each $i=1, \ldots, n$, it is true that
	$\sigma\in H$ if and only if $\sigma(W_{i1})\cap L = V_i$.
	For each $i=1,\ldots, n$, let $\sigma_i \in G$ be the unique element that 
	maps $W_{11}$ to $W_{i1}$. Then for all
	$\sigma \in G$, $\sigma \in  H\sigma_i$ if and only if 
	$\sigma(W_{11})\cap L = V_i$, if and only if $\sigma(W_{11})
	\in \{ W_{i1}, \ldots, W_{id}\}$.  Because $V_1, \ldots, V_n$ are 
	distinct, it follows that $H\sigma_1, \ldots, H\sigma_n$ are exactly 
	all the distinct right cosets of $H$ in $G$. 

	For $\sigma\in G$, $\sigma \in H$ if and only if $\sigma(W_{11})
	\in \{W_{11}, \ldots, W_{1d}\}$, if and only if $a \in \sigma(W_{11})$. 
	Thus, for $\sigma, \tau \in G$, $\tau(a) \in \sigma(W_{11})$ 
	if and only if $\tau^{-1}\sigma \in H$ if and only if $\sigma H = \tau H$.
	So for any root $\tau(a)\in L$ of $f(X)$, where $\tau\in G$,
	we have $\tau(L)=L$, so $\tau H \tau^{-1} = H$, 
	and $\tau H = H\tau$. 
	If $\tau(a)$ were in $V_2$, 
	then $\tau(a)\in \sigma_2(W_{11})$, so $\tau H = \sigma_2H$.
	This means that $\sigma_2\in \sigma_2H = \tau H = H\tau$. Thus 
	$\tau \in H\sigma_2$, whence $\tau(V_1)=
	\tau(W_{11}\cap L) = \tau(W_{11})\cap L = V_2$, contradicting the 
	assumption that for all $\sigma\in \mathrm{Aut}(L/K)$, $\sigma(V_1)\neq V_2$.
	Thus, none of the roots of $f(X)$ in $L$ are contained in $V_2$.
\end{proof}

It is known (see for example Lemma 20.6.3 of \cite{MR2445111}) that if $E$ and $F$
and field extensions of a common subfield $K$ with the property that
every irreducible separable 
polynomial over $K$ has a root in $E$ if and only if it has a root in $F$, 
then $E$ and $F$ are $K$-isomorphic. In the case of valued fields, this is not true
in general. However, under a more stringent condition, there is still a similar result.

\begin{lemma}
	\label{extending-substructure-to-relative-separable-closure}
	Suppose that $(E_1, V_1)$ and $(E_2, V_2)$ are two valued fields 
	extending the valued field $(K,V)$. If every irreducible 
	separable polynomial over $K$
	that has a root in $E_1$ splits both in $E_1$ and $E_2$, 
	and every irreducible separable
	polynomial that has a root in $E_2$ splits both in $E_1$ and
	and $E_2$. Then 
	there is an $\Ldiv$ isomorphism $\sigma: E_1\cap K^\sep 
	\to E_2\cap K^\sep$ (with their induced valuations respectively)
	which restricts to $\mathrm{id}_K$ on $(K,V)$.
\end{lemma}
\begin{proof}
	Since normal extensions are self conjugated, it is harmless
	to slightly abuse the notation. 
	Let $L$ be a finite Galois extension of $K$. 
	
	Suppose that 
	$E_1\cap L=K(a)$. Then the minimal polynomial $f(X)$ 
	of $a$ splits in $E_1$
	and $E_2$.  Thus $E_1\cap L$ is normal over $K$. Suppose $b\in E_2$
	is a root of $f(X)$, then $K(b)$ is also normal over $K$ and 
	$E_1\cap L\cong_K K(b)\subseteq E_2\cap L$.

	By the same reasoning, $E_2\cap L$ is $K$-isomorphic to a subfield of 
	$E_1\cap L$. This means that $[K(b):K]\leq [E_2\cap L:K]\leq [K(b):K]$. 
	Thus $K(b)=E_2\cap L$.  

	Therefore $E_1\cap L$ is $K$-isomorphic to $E_2\cap L$, both of which are 
	normal over $K$. By Fact \ref{conjugate-theorem-normal-extensions}, 
	there is an $\Ldiv$ isomorphism $\sigma_L: E_1\cap L
	\to E_2\cap L$ (with their induced valuations respectively)
	which restricts to $\mathrm{id}_K$ on $(K,V)$.

	Thus, so far we have show that for each finite Galois extension $L$, 
	the set 
	$$A_L:=\{\sigma_L:E_1\cap L\to E_2\cap L 
	\text{ an $\Ldiv$-isomorphism that restricts to $\mathrm{id}_K$ }\}$$
	is not empty. Thus the inverse limit 
	$$ H:=\lim_{\longleftarrow \atop L } A_L,$$
	taken over all $L$ finite Galois over $K$, is not empty. Any element 
	$\sigma \in H$ is 
	an $\Ldiv$ isomorphism $\sigma: E_1\cap K^\sep 
	\to E_2\cap K^\sep$ (with their induced valuations respectively)
	which restricts to $\mathrm{id}_K$ on $(K,V)$.
\end{proof}

\begin{lemma}\label{algebraically-independent-amlgmation}
	Suppose that $(E_1, V_1)$, $(E_2, V_2)$ and $(F, W)$ are three valued-field
	extension of a valued field 
	$(K,V)$, where $(F,W)$ is a $(|E_1|\times|E_2|)^+$-saturated 
	model of $\ACVF$. Then there exist $\Ldiv$-embeddings $\sigma_1:E_1\to F$ 
	and $\sigma_2: E_2\to F$ such that:
	\begin{enumerate}
		\item the restriction of $\sigma_1$ and $\sigma_2$ on $K$ are 
			the identity map of $K$;
		\item $\sigma_1(E_1)$ and $\sigma_2(E_2)$ are algebraically 
			independent over $K$.
	\end{enumerate}
\end{lemma}
\begin{proof}
	By Fact \ref{ACVF-QE}, $\ACVF$ admits quantifier elimination, therefore, 
	by Fact \ref{use-embedding-to-get-QE}, we can find two $\Ldiv$-embeddings
	$f_1: E_1\to F$ and $f_2: E_2\to F$, such that both $f_1$ and $f_2$ 
	are $\mathrm{id}_K$ when restricted to the subfield $K$. Take 
	$\{e_i\}_{i\in I}$ to be a transcendence basis of $E_1$ over $K$. 
	Since $F$ is 
	$(|E_1|\times|E_2|)^+$-saturated, there exists a set
	of elements $\{\alpha_i\}_{i\in I}$ in $F$ 
	algebraically independent over the compositum $L:=f_1(E_1).f_2(E_2)$
	with each of the $w(\alpha_i)$ larger than any element in $w(L^\times)$.
	It then follows that $w(\alpha_i + f_1(e_i))= w(f_1(e_i))$ for all $i\in I$ 
	and 
	the set  $\{\alpha_i + f_1(e_i)\}_{i\in I}$ is algebraically 
	independent over $L$. This in turn 
	implies that for any multivariate polynomial $h$ with 
	coefficients in $f_1(E_1)$, $w(h(\alpha+f_1(e_i))_{i\in I})= 
	w(h(f_1(e_i))_{i\in I})$.
	Thus mapping $f_1(e_i)$ to $\alpha_i + f_1(e_i)$, 
	we obtain an $\Ldiv$-embedding of $K(f_1(e_i))_{i\in I}$ to 
	$F$, which extends to an $\Ldiv$-embedding (by quantifier elimination again)
	$g: f_1(E_1) \to F$ which necessarily has the 
	property that $g(f_1(E_1))$ is algebraically independent 
	from $f_2(E_2)$ over $K$. Thus, $\sigma_1:=g\circ f_1$ and $\sigma_2:=f_2$
	satisfy the desired properties. 
\end{proof}
\begin{lemma}
	\label{lemma-downward-morphism-from-a-regular-extension}
	Suppose that $K$ is a PAC field with a non-trivial valuation $V$, 
	$(L,W)$ is a valued field extension of $(K,V)$ where $L/K$ is regular.
	Suppose that $E$ is a subfield of $L$ contained in a $K$-subalgebra 
	$R$ of $L$, where $R$ is generated by not more that $|E|$-elements over $K$.
	Suppose that $(K,V)$ is $|E|^+$-saturated in $\Ldiv$. Then 
	there is a ring homomorphism $\phi: R\to K$ satisfying the following 
	conditions simultaneously:
	\begin{enumerate}
		\item $\phi$ restricts to the identity map on $K$;
		\item $\phi$ restricts to an $\Ldiv$-embedding of $(E,W\cap E)$
			into $(K,V)$.
	\end{enumerate}
\end{lemma}
\begin{proof}
	We may assume that $E\backslash K$ is 
	not empty. Let $k$ be the prime field of $K$.
	By assumption, $R=K[x_i]_{i\in J}$, where $J$ is an index set of 
	cardinality $|E|$ and $\{x_i\}_{i\in J} \subseteq E\backslash K$. 
	Then for any finite subset 
	$S\subseteq J$, $K(x_i)_{i\in S}$ is a regular extension of $K$. This 
	means that $(x_i)_{i\in S}$ is a generic point of
	an absolutely affine algebraic set $V_S$ defined over $K$. Let $k_S$ be 
	a finitely generated (over $k$) subfield (of $K$) of definition of $V_S$
	and
	$(\gamma_j)_{j\in J_S}$ be a finite set of generators for $k_S/k$. 
	Let $C$ be the compositum of all the fields of definition $k_S$. Then 
	$|C|\leq |E|$.
	Let $I_S$ be the vanishing ideal of $(x_i)_{i\in S}$ in $K[X_i]_{i\in S}$. 
	Then $I_S$ is finitely generated by elements in $C[X_i]_{i\in S}$, 
	let $(f_l)_{l\in L_S}$ be a finite set of generators.
	
	In the following, $\vec{x}$ denotes a tuple from the set 
	$\{x_i\}_{i\in J}$. 
	Let $\Sigma$ be the following set of formulas in free variables 
	$\{X_j\}_{j\in J}$
	\begin{align*}
		&\left(\cup_{S\subset J \atop \text{finite}}
		\{f_l(\vec{X}')=0\}_{l\in L_S} 
	\right)
		\bigcup \{ v(p(\vec{X}))\geq v(q(\vec{X})) \mid p,q\in C[\vec{X}], 
		\\
		& \hskip12em v(p(\vec{x}))\geq v(q(\vec{x})),
	\vec{x} \text{ a tuple in } E\backslash K\},
	\end{align*}
	where $\vec{X}'$ and $\vec{X}$ are not necessarily the same tuple. 
	Then $\Sigma$ is a set of formulas in $|E|$ variables over a set 
	of cardinality not more than $|E|$.
	
	We show that $\Sigma$ is realizable in $(K, V)$. It is enough to show that 
	$\Sigma$ is finitely realizable in $(K, V)$, because $(K,V)$ is 
	$|E|^+$-saturated. Let 
	\begin{equation}
		\label{eq-1}
		f_{l_1}, \ldots, f_{l_t}, v(q_1)\geq v(p_1), \ldots, v(p_s)\geq v(q_s)
	\end{equation}
	be finitely many elements arbitrarily chosen from $\Sigma$. 
	Then there exist a tuple $\vec{x}$ realizing all the formulas in 
	Equation \eqref{eq-1}. 
	The indexes of $\vec{x}$ corresponds to a finite subset 
	$S_0\subseteq J$. Then the absolutely irreducible affine algebraic set 
	$V_{S_0}$ is defined over $C$, hence over $K$. We may also assume that 
	$$p_1(\vec{x}), \ldots, p_s(\vec{x}), q_1(\vec{x}), \ldots, q_s(\vec{x})$$
	are all non-zero, otherwise say $p_i(\vec{x})=0$,
	then $v(p_i)\geq v(q_i)$ is equivalent to 
	an $\Lr$-formula and could be treated as one of $f_{l_1}, \ldots, f_{l_t}$.
	
	Consider the formula
	\begin{equation}
		\label{eq-2}
		\Phi:=\left[\wedge_{l\in l_{S_0}} f_l(\vec{X})=0\right]
	\bigwedge \left[\wedge_{i=1}^s v(p_i(\vec{X}))\geq v(q_i(\vec{X}))\right].
	\end{equation}
	$\vec{x}$ realizes $\Phi$ in $(L^\alg, W)$. Since $\ACVF$ is model complete
	in $\Ldiv$, there is also a realization of $\Phi$ in $(K^\alg, V)$, denoted
	by $\vec{a}$. By Fact \ref{Kollar-density}, there is a point $\vec{b}$ 
	with coordinates in $K$ such that $\vec{b}$ is sufficiently close
	to $\vec{a}$ with respect to the topology induced by $W$.  This 
	means that $\vec{b}$ also realizes $\Phi$ in $K$. Since $\vec{b}$ is 
	a point in $V_{S_0}$, it satisfies any algebraic relation satisfied by 
	the generic point $\vec{x}$ over $C$. Thus $\vec{b}$ also satisfies 
	all the formulas in Equation \eqref{eq-1}. 
	This proves that $\Sigma$ is realizable
	in $K$.
	
	Let $\{\vec{y}_i\}_{i\in J}$ be a realization of $\Sigma$ in $K$. 
	Then mapping $x_j \to y_j$, with $\mathrm{id}_K:K\to K$, we get a 
	natural map from $R$ to $K$, denoted by $\phi:R\to K$. Since $\phi$
	preserves all algebraic relations among the $x_j$, it is a ring homomorphism
	which restricts to the identity map on $K$. Meanwhile, $E$ is a field, 
	so the restriction of $\phi$ to $E$ is an isomorphism, which also 
	preserves the valuation on $E$; therefore $\phi$ is an $\Ldiv$-embedding 
	of $(E,W\cap E)$ into $(K,V)$.
\end{proof}

The following lemma is similar to its field theoretic version, as in Proposition 
11.4.1 of \cite{MR2445111}. The proof is also similar. Thus we will not present 
the full version of the proof, but refer the reader to original proof 
of said proposition. 
\begin{lemma}
	\label{lemma-downward-morphism-from-a-regular-extension-preserving-p-indep}
	Suppose that $K$ is a PAC field with a non-trivial valuation $V$, 
	$(L,W)$ is a valued field extension of $(K,V)$ where $L/K$ is regular.
	Suppose that $E$ is a subfield of $L$ contained in a $K$-subalgebra 
	$R$ of $L$, where $R$ is finitely generated over $K$. 
	Suppose furthermore that the exponent 
	of imperfection of $K$ is $d\in \omega\cup\{\infty\}$, that
	$y_1, \ldots, y_m\in R$ are $p$-independent in $L$,  $m\leq d$,
	and that $(K,V)$ is $|E|^+$-saturated in $\Ldiv$.
	Then 
	there is a ring homomorphism $\phi: R\to K$ satisfying the following 
	conditions simultaneously:
	\begin{enumerate}
		\item $\phi$ restricts to the identity map on $K$;
		\item $\phi$ restricts to an $\Ldiv$-embedding of $(E,W\cap E)$
			into $(K,V)$;
		\item $\phi(y_1), \ldots, \phi(y_m)$ are $p$-independent
			in $K$.
	\end{enumerate}
\end{lemma}
\begin{proof}
	By the proof of Proposition 11.4.1 of \cite{MR2445111}, 
	we may assume that $M:=L^pK(y_1, \ldots, y_m)=L^pK(y_1, \ldots, y_k)$
	where $y_1, \ldots, y_k$ are $p$-independent over $L^pK$; there 
	exist $a_1, \ldots, a_n\in K$ with $m\leq n$ such that they are
	$p$-independent over $K^p$ and 
	$$ a_{n-m+k+1}, \ldots, a_n \in 
	L^p(a_1, \ldots, a_{n-m+k}, y_1, \ldots, y_m).$$
	Let  $\eta_i:=(a_{n-m+i}y_i)^{1/p}$, $i=1, \ldots, k$ and $L'=L(\eta_1, 
	\ldots, \eta_k)$. Then by the proof op.~cit., $L'/K$ is a regular 
	field extension.

	Therefore, there exists elements $v_{ij}\in L$ such that for 
	$i=n-m+k+1, \ldots, n$, 
	\begin{equation}
		\label{eq-4}
		a_i=\sum_{j_1, \ldots, j_{n+k}} 
	v_{ij}^p a_1^{j_1}\cdots a_{n-m+k}^{j_{n-m+k}}
	y_1^{j_{n-m+k+1}}\cdots y_m^{j_{n+k}}.
	\end{equation}
	Let 
	$$ S:=R[\eta_1, \ldots, \eta_k, y_1^{-1}, \ldots, y_m^{-1}, v_{ij}]_{ij}.$$
	Then the quotient field of $S$ is $L'$, which is regular over $K$.
	Thus by Lemma \ref{lemma-downward-morphism-from-a-regular-extension}, 
	there exists a ring homomorphism $\phi:S\to K$ such that $\phi\mid_{K}=
	\mathrm{id}_K$ and $\phi\mid_{E}$ is an $\Ldiv$-embedding of $(E, W\cap E)$
	into $(K,V)$. Then $\phi$ satisfies
	$$
		\phi(\eta_i)^p = a_{n-m+i}\phi(y_i), \quad i=1, \ldots, k.
	$$
	Thus 
	\begin{equation}
		\label{eq-5}
		a_{n-m+i}=\phi(\eta_i)^p\phi(y_i)^{-p}\phi(y_i)^{p-1},
		\quad i=1, \ldots, k.
	\end{equation}
	Apply $\phi$ to Equation \eqref{eq-4}, we get
	\begin{equation}
		\label{eq-6}
		\begin{aligned}
			&a_i=\sum_{j_1,\ldots, j_{n+k}}
			\phi(v_{ij})^p a_1^{j_1}\cdots a_{n-m+k}^{j_{n-m+k}}
			\phi(y_1)^{j_{n-m+k+1}}\cdots \phi(y_m)^{j_{n+k}}, \\
			&\hskip14em i=n-m+k+1, \ldots, n.
		\end{aligned}
	\end{equation}
	Thus by Equation \eqref{eq-5} and Equation \eqref{eq-6}, 
	we get 
	$$ a_1, \ldots, a_n\in K^p(a_1, \ldots, a_{n-m}, \phi(y_1), \ldots, 
	\phi(y_m)). $$
	But $a_1, \ldots, a_n$ are $p$-independent over $K^p$, so 
	$$ a_1, \ldots, a_{n-m}, \phi(y_1), \ldots, \phi(y_m)$$
	are also $p$-independent over $K^p$.  Thus $\phi$ satisfies all the 
	desired properties.
\end{proof}

Combining Lemma 
\ref{lemma-downward-morphism-from-a-regular-extension-preserving-p-indep}
and the proof of Lemma 
\ref{lemma-downward-morphism-from-a-regular-extension}, we can improve the later to 
preserve the $p$-independence relation.
\begin{lemma}
	\label{lemma-downward-morphism-from-a-regular-extension-preserving-p-indep-saturated}
	Suppose that $K$ is a PAC field with a non-trivial valuation $V$, 
	$(L,W)$ is a valued field extension of $(K,V)$ where $L/K$ is regular.
	Suppose that $E$ is a subfield of $L$ contained in a $K$-subalgebra 
	$R$ of $L$, where $R$ is generated by not more than $|E|$ elements over $K$. 
	Suppose furthermore that the exponent of imperfection 
	of $K$ is $d\in \omega\cup\{\infty\}$,
	$\{y_m\}_{m\in M}\subseteq R$ are $p$-independent in $L$, where 
	$|M|\leq \min\{|E|, d\}$.
	Suppose that $(K,V)$ is $|E|^+$-saturated in $\Ldiv$.
	Then 
	there is a ring homomorphism $\phi: R\to K$ satisfying the following 
	conditions simultaneously:
	\begin{enumerate}
		\item $\phi$ restricts to the identity map on $K$;
		\item $\phi$ restricts to an $\Ldiv$-embedding of $(E,W\cap E)$
			into $(K,V)$;
		\item $\{\phi(y_j)\}_{j\in J}$ are $p$-independent
			in $K$.
	\end{enumerate}
\end{lemma}
\begin{proof}
	The strategy of the proof is very much the same as the one for Lemma
	\ref{lemma-downward-morphism-from-a-regular-extension}. 
	Using the notation from that proof, assume that $y_m=g_m(\vec{x})$. 
	Instead of 
	showing the set of formulas $\Sigma$ is finitely realizable in $K$, 
	we show that 
	$$ \Sigma\cup\{\text{$g_m(\vec{X})$ are $p$-independent}\}_{m\in M}$$
	is finitely realizable in $K$, which is guaranteed by Lemma 
	\ref{lemma-downward-morphism-from-a-regular-extension-preserving-p-indep}.
\end{proof}

\section{Existence of various pseudo-algebraically closed fields}

In the manuscript \cite{CherlinvddMacintyre:ThRCfields}, 
the elementary theory of regularly closed fields were studied, where 
a field is called a regularly closed field exactly when it is 
pseudo-algebraically closed. The following fact
was proved in \cite{CherlinvddMacintyre:ThRCfields}.

\begin{fact}
	\label{existence-of-e-free-PAC-with-exponent-of-imperfection-d}
	[Proposition 38 of \cite{CherlinvddMacintyre:ThRCfields}]
	Let $d$ be any element in $\omega\cup\{\infty\}$ denoting a chosen 
	exponent of imperfection.
	Let $F$ be a field algebraic over its prime field, $G$ a projective
	profinite group and $\pi: G\to \Gal(F)$ an epimorphism. Then there 
	exists a field extension $K/F$ satisfying all the following conditions:
	\begin{enumerate}
		\item $K$ is pseudo-algebraically closed;
		\item the relative algebraically closure of the prime field 
			in $F$ is exactly $K$;
		\item there exists an isomorphism $\sigma : G\to \Gal(K)$ 
			such that the following diagram commutes;
			$$
			\xymatrix{
				G \ar[rr]^\sigma \ar[dr]_{\pi}
				& & \Gal(K) \ar[dl]^{\mathrm{res}}\\
				& \Gal(F) & \\
				}
			$$
		\item the exponent of imperfection of $K$ is $d$.
	\end{enumerate}
\end{fact}

Since every PAC field is infinite, and on any infinite field, there is 
always a non-trivial valuation, using the proposition above, we get that 
there exists a PAC field $K$ with a non-trivial valuation satisfying 
all the conditions in said proposition.  In the rest of the article, 
we will be considering various first order theories of pseudo-algebraically 
closed fields with non-trivial valuation, Proposition 
\ref{existence-of-e-free-PAC-with-exponent-of-imperfection-d} guarantees that
these theories are all consistent. 

\section{Lack of Quantifier elimination}
\label{section-lack-of-QE}
Following the notation in \cite{MR2445111} below, 
we use $\LR$ to denote the language $\Lr\cup\{R_n\}_{n\geq 1}$, 
$\LRdiv$ to denote the language $\Ldiv\cup\{R_n\}_{n\geq 1}$,
and $\LRpdiv$ to denote the language $\Lpdiv\cup\{R_n\}_{n\geq 1}$, where $R_n$ 
is an $n$-ary predicate. 
If $T$ is a theory of fields 
in any of $\Lr$, $\Ldiv$ or $\Lpdiv$, then we use $T_R$
to denote the theory 
$T\cup\{\phi_n\}_{n\geq 1}$ in $\LR$, $\LRdiv$ or $\LRpdiv$ respectively,
where $\phi_n$ is the axiom
$$ R_n(a_0, a_1, \ldots, a_{n-1}) \leftrightarrow 
\exists x ( x^n + a_{n-1}x^{n-1} + \cdots + a_{1} x + a_0=0).$$

\begin{theorem}
	Suppose that $T$ is one of the following theories, then $T_R$
	does NOT have quantifier elimination.
	\begin{enumerate}
		\item The theory of pseudo-algebraically closed non-trivially 
			valued fields, in $\Ldiv$ or in $\Lpdiv$ (characteristic $p$);
		\item the theory of $e$-free ($e\in \omega\backslash\{0\}$ fixed)
			pseudo-algebraically closed
			non-trivially valued fields, of characteristic $p$,
			with exponent of imperfection $d$ 
			($d\in \omega\cup\{\infty\}$ fixed), in $\Lpdiv$;
		\item the theory of $\omega$-free pseudo-algebraically closed 
			non-trivially valued fields, of characteristic $p$,
			with exponent of imperfection $d$
			($d\in \omega\cup\{\infty\}$ fixed), in $\Lpdiv$.
	\end{enumerate}
\end{theorem}
\begin{proof}
	We only prove the cases of $(2)$ and $(3)$. Case $(1)$ 
	is implied by the proof of Case $(2)$.

	First, for Case $(2)$.

	For a fixed $e\in \omega\backslash\{0\}$
	and a fixed $d\in \omega\cup\{\infty\}$, let $T$ denote
	that theory of $e$-free pseudo-algebraically closed non-trivially 
	valued fields with degree of imperfection $d$.
	To show that $T_R$ does not have 
	quantifier elimination, by Fact \ref{substructure-complete}
	it is enough to show that $T_R$ is not substructure complete.

	Let $G=
	\hat{F}_e\coprod \hat{F}_e\cong\hat{F}_{2e}$ be the free profinite product of 
	the free profinite group of rank $e$ with itself.
	Let $H$ be the first free factor of $G$, then by Fact 
	\ref{selfnormalizing-component-of-free-product}, $H$ self-normalizing in $G$.
	By Fact 
	\ref{existence-of-e-free-PAC-with-exponent-of-imperfection-d}, 
	there exists an $2e$-free PAC field with exponent of imperfection $d$ 
	whose absolute Galois group is $G$; let 
	$K$ be such a field.  
	Let $E$ be the fixed field of $H$ in $K^\sep$. 
	Because $H$ is self-normalizing,  if $\sigma \in \mathrm{Aut}(E/K)$, 
	then $\sigma$ extends to an element $\tilde{\sigma}\in G$ satisfying
	$\tilde{\sigma} H \tilde{\sigma}^{-1} = H$, which 
	in turn implies that $\tilde{\sigma} \in H$, whence $\tilde{\sigma}$
	fixes $E$ element-wise. Thus $\mathrm{Aut}(E/K)=1$. 
	Thus by Fact \ref{algebraic-extensions-of-PAC} and 
	Fact \ref{separable-algebraic-extension-imperfection-degree}, 
	$E$ is an $e$-free
	PAC field whose exponent of imperfection is also $d$.

	Because $H$ is self-normalizing in $G$, the field extension $E/K$
	is not Galois. This means that there exists a finite separable normal
	extension $L/K$ contained in $K^\sep$ such that $L\cap E$ is not 
	normal over $K$. 
	
	Pick a non-trivial valuation $V$ on $K$, by Lemma 
	\ref{number-of-extensions-equals-degree}, 
	there are exactly $[L\cap E:K]$ distinct extensions of $V$ on $K$ to
	$L\cap E$, 
	denoted  by  $V_1, \ldots, V_n$. Because $L\cap E$ is not Galois over $K$, 
	there exists two valuations, $V_i$ and $V_j$, not $K$-conjugated in $L\cap E$.
	By Lemma \ref{distinguishing-nonconjugated-valuations}, 
	there exists a primitive element $a$ for the extension $L\cap E/K$ 
	with minimal polynomial $f(X)\in K[X]$ 
	such that $a\in V_i\backslash V_j$ and none of the roots 
	of $f(X)$ in $L\cap E$ is in $V_j$. Let $\Phi$ be an $\Ldiv(K)$ sentence
	saying that $f(x)$ has a root with a non-negative value. Then 
	$(L\cap E, V_1)\models \Phi$ but $(L\cap E, V_2)\not\models \Phi$.
	Since $L$ is normal over $K$, all the roots of $f(X)$ in $E$ are 
	already in $L\cap E$, this means that if we denote extensions of 
	$V_i$ and $V_j$ to $E$ by $\tilde{V}_i$ and $\tilde{V}_j$ respectively, 
	then 
	$$ (E, \tilde{V}_i)\models \Phi, 
	\qquad \text{but} \qquad 
	(E, \tilde{V}_j)\not\models\Phi.$$
	Therefore $(E, \tilde{V}_i)\not\equiv_{(K,V)} (E,\tilde{V}_j)$.
	Thus $T_R$ is not substructure complete.

	Now, for Case $(3)$. The proof is similar to the above, exact that 
	one replaces the group $G$ above
	by $G=\hat{F}_\omega \coprod \hat{F}_\omega $ and use Fact
	\ref{free-product-of-projective-groups}.
\end{proof}

In view of the theorem above, in order to obtain quantifier elimination, we need
to functionalize the predicates $R_n$. On obvious possible solution is to 
uniformly pick out the roots of all the polynomials with roots; but this 
does not seem easy if possible at all to do.
Since the theorem above suggests that 
one obstruction to having quantifier elimination is the existence of 
non-Galois extensions, we introduce the following function symbols to 
eliminate these non-Galois extensions. 

\begin{definition}
	Let $K$ be any field. Given a  polynomial $f(X)\in K[X]$, 
	we define the {\bf maximal splitting factor} of $f(X)$ to the 
	unique monic polynomial $g(X)$ satisfying the following two conditions:
	\begin{enumerate}
		\item $g(X)$ splits into linear factors over $K$;
		\item there exists an polynomial $h(X)\in K[X]$ without roots in $K$
			such that $f(X)=g(X)h(X)$.
	\end{enumerate}
	This is equivalent to saying that $g(X)$ is the product 
	$\prod_{i=1}^r (x-r_i)$, where $r_1, \ldots, r_i$ are all the roots of 
	$f(X)$ in $K$.
\end{definition}

\begin{definition}
	Let $K$ be any field and $n$ be an positive integer. Given $n+1$ 
	elements $a_0, \ldots, a_{n-1}, a_n\in K$, 
	denote the polynomial 
	$$a_nX^n+a_{n-1}X^{n-1} + \cdots + a_{1}X+a_0$$
	by $\pi_n(a_0, \ldots, a_{n-1}, a_n)$.  Suppose that the maximal splitting 
	factor of $\pi_n(a_0, \ldots, a_{n-1}, a_n)$ over $K$ is 
	$$g(X) = b_n X^n + b_{n-1}X^{n-1} + \cdots + b_1X+b_0,$$
	where $b_n, \ldots, b_0 \in K$, some of which could be zeros.
	For each $0\leq i\leq n$, we define the $i$-th {\bf 
	splitting coefficient 
	}
	$\theta_{n,i}(a_0, a_1, \ldots, a_n)$ by dictating that
	$$\theta_{n,i}(a_0, a_1, \ldots, a_n)=b_i.$$
\end{definition}
\begin{definition}
	We use $\Ltheta$ to denote the language $\Lr\cup\{\theta_{n,i}\}_{n\geq 1, 
	0\leq i\leq n}$, $\Lthetadiv$ to denote the language 
	$\Ldiv\cup\{\theta_{n,i}\}_{n\geq 1, 0\leq i\leq n}$, 
	and $\Lthetapdiv$ to denote the language 
	$\Lpdiv\cup\{\theta_{n,i}\}_{n\geq 1, 0\leq i\leq n}$. 
	A field naturally interprets the symbols $\theta_{n,i}$ as 
	splitting coefficients.
\end{definition}

It will be proved in Section \ref{Section-quantifier-elimination} that 
several theories of pseudo-algebraically closed fields with non-trivial valuation 
do in fact have quantifier elimination in $\Lthetapdiv$. But it is worth pointing 
out that 
in $\Lthetapdiv$, 
the theory of pseudo-algebraically closed fields with non-trivial valuations
itself as a total, does not admit quantifier elimination.

\begin{theorem}
	The following statements are true.
	\begin{enumerate}
		\item Let $\PACVF$ be the theory, 
			in $\Lthetadiv$, of pseudo-algebraically
			closed non-trivially valued fields.  
			Then $\PACVF$ does NOT admit quantifier elimination. 
		\item Let $\PACVF_{p,d}$ 
			be the theory, in $\Lthetapdiv$, of pseudo-algebraically
			closed non-trivially valued 
			fields, of characteristic $p$, with exponent of 
			imperfection 
			$d$ ($d\in \omega\cup\{\infty\}$). 
			Then $\PACVF_{p,d}$ does NOT 
			admit quantifier elimination either. 
	\end{enumerate}
\end{theorem}

\begin{proof}
	It is enough to show that $\PACVF_{p,d}$ is not substructure complete. 
	The obstruction occurs already on the field theoretic level, not necessary
	to involve the valuations. 

	The $K$ be a countable separably closed field of characteristic $p$ with 
	exponent of imperfection $d$. Let $t$ be transcendental element over $K$. Then
	$K(t)$ is a countable Hilbertian field and the extension $K(t)/K$ is regular.
	Let $V$ be any non-trivial valuation on $K(t)$
	By Fact \ref{PAC-Nullstellensatz}  and
	Fact \ref{so-called-The-Free-Generator-Theorem}, 
	there exists a field $E_1/K(t)$ which is $1$-free PAC 
	and there exists a field $E_2/K(t)$ which is $2$-free PAC. 
	Let $V_1$ and $V_2$ be any two extensions of $V$ to $E_1$ and $E_2$
	respectively. 
	Since $E_1/K$ and $E_2/K$ are both separable and $K$ is separably closed,
	the $\Lthetapdiv$-substructure on $K$ in $E_1$ and the 
	$\Lthetapdiv$-substructure on $K$ in $E_2$ agree. 
	Thus 
	$$ (E_1, V_1) \not\equiv_{(K,V)} (E_2,V_2),$$
	because they have different free absolute Galois groups. 
\end{proof}

In view of the theorem above, in order to obtain quantifier elimination, we 
need to impose conditions on the Galois groups of the models. This leads us 
to obtain the results in Section \ref{Section-quantifier-elimination}.

\section{The Embedding Lemma}

The so-called Embedding Lemma (Lemma 20.2.2 in \cite{MR2445111}) plays 
an important role in the analysis of the model theory of PAC fields. Analogously, 
we also have a valuation theoretic version of the Embedding Lemma, which plays
a similar important role in the model theoretic analysis of the pseudo-algebraically
closed valued fields.

\begin{lemma}
	[Valuation Theoretic Embedding Lemma]
	\label{valuation-theoretic-embedding-lemma}
	Suppose that valued field $(E, V_E)$ extends $(L,V_L)$  and
	that valued field $(F,V_F)$ extends 
	$(M,V_M)$, where $E/L$ and $F/M$ are regular field extensions.
	Suppose that $F$ is PAC and 
	$|E|^+$-saturated in $\Ldiv$ and that the exponent of imperfection 
	of $E$ is not
	more than the exponent of imperfection of $F$. 
	Suppose that there is a field isomorphism $\Phi_0:L^\sep \to M^\sep$ 
	and a commutative diagram
	\begin{equation}
		\xymatrix{
			\Gal(E) \ar[d]_{\mathrm{res}} & \Gal(F) 
			\ar[l]_{\varphi} \ar[d]^{\mathrm{res}}
			 \\
			\Gal(L) & \Gal(M)\ar[l]^{\varphi_0},
			}
	\end{equation}
	where $\varphi_0$ is the isomorphism induced by $\Phi_0$ (namely 
	$\varphi_0(\sigma)= \Phi_0^{-1}\circ \sigma \circ \Phi_0$), 
	$\varphi$ is a homomorphism, and $\Phi_0$ restricts to an $\Ldiv$-isomorphism
	from $(L, V_L)$ onto $(M, V_M)$. 

	Then there exists an extension of $\Phi_0$ 
	to a field embedding $\Phi:E^\sep\to F^\sep$ that induces $\varphi$ with 
	$F/\Phi(E)$ separable and that $\Phi$ restricts to an $\Ldiv$-embedding 
	from $(E,V_E)$ into $(F,V_F)$. If furthermore $\varphi$ is surjective, 
	then $F/\Phi(E)$ is regular.
\end{lemma}
\begin{proof}
	The proof is similar to that of Lemma 20.2.2 in \cite{MR2445111}.

	We may identify $L$ and $M$, and assume that both of 
	$\Phi_0$ and $\varphi_0$ are 
	the identity maps. By Lemma \ref{algebraically-independent-amlgmation}, 
	we may assume that $E$ is algebraically independent from $F$ over $L$.
	Then by the proof op.~cit., $E^\sep/L^\sep$ and $E^\sep/L^\sep$ are both 
	regular, and $EF$ is regular both over $E$ and over $F$. 
	Thus $E^\sep$ is linearly disjoint from $F^\sep$ over $L^\sep$. Meanwhile, 
	each $\sigma \in \Gal(F)$ extends uniquely to a field 
	automorphism $\tilde{\sigma}\in 
	\Gal(E^\sep F^\sep/EF)$ with 
	$$ \tilde{\sigma}x = 
	\begin{cases}
		\varphi(\sigma) x, & x\in E^\sep, \\
		\sigma x, & x\in F^\sep.
	\end{cases}
	$$
	Thus the map $\sigma \mapsto \tilde{\sigma}$ embeds $\Gal(F)$ into 
	$\Gal(E^\sep F^\sep/EF)$. Let $D$ be the fixed field of the image of 
	$\Gal(F)$ under the aforementioned map.  Then 
	$D\cap F^\sep=F$, $DF^\sep = E^\sep F^\sep$, and 
	$E^\sep\subseteq E^\sep F^\sep
	= D[F^\sep] =F^\sep[D]$. For each $x\in E^\sep$, choose $y_i \in F^\sep$
	and $d_i\in D$ where $i$ ranges over a finite set $I_x$ such that 
	$x= \sum_{i\in I_x} y_id_i$. Let $D_0=E\cup\{d_i\mid x\in E^\sep, i\in I_x\}$.
	Then $|D_0| = |E|$ and $E^\sep \subseteq F^\sep[D_0]$.

	If $S$ is a $p$-basis of $E$, then since $F(D_0)/E$ is separable, $S$
	remains $p$-independent over $F(D_0)^p$. Therefore by Lemma
	\ref{lemma-downward-morphism-from-a-regular-extension-preserving-p-indep-saturated}
	there exists a ring homomorphism $\Psi: F[D_0] \to F$ satisfying the 
	following three conditions:
	\begin{enumerate}
		\item $\Psi$ restricts to the identity map on $F$;
		\item $\Psi$ restricts to an $\Ldiv$-embedding of $(E, V_E)$
			into $(F,V_F)$;
		\item $\Psi(S)$ is $p$-independent in $F$.
	\end{enumerate}
	From $(3)$, we know that $F/\Psi(E)$ is separable. Since $D$ is linearly 
	disjoint from $F^\sep$ over $F$, $\Psi$ extends to a ring homomorphism
	$\tilde{\Psi}:F^\sep[D_0] \to F^\sep$ such that $\tilde{\Psi}$ 
	restricts to the identity map on $F^\sep$. Furthermore, for each $\sigma\in 
	\Gal(F)$, and for each $x\in F^\sep\cup D_0$,
	we have 
	$$ \tilde{\Psi}(\tilde{\sigma}x) = \sigma \tilde{\Psi}(x),$$
	which then holds for each $x\in F^\sep[D_0]\supseteq E^\sep$. 
	Thus if we let $\Phi = \tilde{\Psi}\mid_{E^\sep}$, then $\Phi$ satisfies 
	all the desired properties.
	
	If $\varphi$ is surjective, then we need to show that $\Phi(E)$ is 
	relatively algebraically closed in $F$ so that $F/\Phi(E)$ is regular.
	For any $y\in F$ algebraic over $\Phi(E)$,  there exists a unique
	$x\in E^\sep$ such that $\Phi(x)=y$. So for any $\sigma \in \Gal(F)$, we 
	have 
	$ \tilde{\Psi}(\tilde{\sigma} x) = \sigma \tilde{\Psi}(x)=\sigma y = y.$ 
	But $\tilde{\sigma} x = \varphi(\sigma)x \in E^\sep$, 
	so $\varphi(\sigma) x = \tilde{\sigma} x = x$. Since $\varphi$ is surjective, 
	this means that $\{\varphi(\sigma)\}_{\sigma \in \Gal(F)} = \Gal(E)$, 
	thus $x\in E$. So $y\in \Phi(E)$. In conclusion, $F/\Phi(E)$ is indeed regular
	if $\varphi$ is surjective.
\end{proof}

\section{Quantifier elimination}
\label{Section-quantifier-elimination}
\begin{definition}
	[Definition 24.1.2 of \cite{MR2445111}]
	A profinite group $G$ has the {\bf Embedding Property} if each 
	embedding problem $(\zeta: G\to A, \alpha : B \to A)$ where 
	$\zeta$ and $\alpha$ are epimorphisms and $B\in \Im(G)$ is 
	solvable, that is, exists an epimorphism $\gamma : G\to B$ 
	such that the following digram commutes
	$$
	\xymatrix{
		& G \ar[dl]_{\gamma} \ar[d]^{\zeta} \\
		B \ar[r]_{\alpha} & A
		}
	$$
\end{definition}
\begin{definition}
	[Definition 24.1.3 of \cite{MR2445111}]
	A field $K$ is called a {\bf Frobenius field} if $K$ is PAC 
	and $\Gal(K)$ has the Embedding Property.
\end{definition}
\begin{remark}
	In \cite{CherlinvddMacintyre:ThRCfields}, the Frobenius fields are 
	called Iwasawa regularly closed fields, because the Embedding Property
	was first pointed out by Iwasawa. These fields are called Frobenius 
	fields by various authors because these fields admit an analog of 
	the Frobenius Density Theorem in number theory. In this article, we adopt 
	the Frobenius-field terminology, which seems to be more widely adopted 
	by other authors.
\end{remark}

In \cite{CherlinvddMacintyre:ThRCfields}, the theory of Frobenius fields was 
studied and a quantifier elimination result was proved. In order to state their 
result, we introduced their so-called {\it Galois formalism}.
\begin{definition}
	[see Page 60 on \cite{CherlinvddMacintyre:ThRCfields}]
	The {\bf Galois formalism} is the language of rings $\Lr$ expanded 
	by the predicates $R_n$ mentioned in Section \ref{section-lack-of-QE} 
	and a $0$-place predicate $I_G$, for each isomorphism type $G$ of
	finite groups.
\end{definition}
Note that over a field,  all the new predicates in the Galois formalism are 
in fact equivalent to $\Lr$-sentences.
\begin{fact}
	[Section 2.10 of \cite{CherlinvddMacintyre:ThRCfields}]
	Being Frobenius is a first order property in the language of rings $\Lr$.
\end{fact}
\begin{fact}
	[Theorem 41 of \cite{CherlinvddMacintyre:ThRCfields}]
	\label{Perfect-FrobF-admits-QE}
	Let $\mathrm{FrobF}_{p^0}$ be the theory of perfect Frobenius fields  in the 
	Galois formalism, where the predicates $R_n$ are interpreted as in Section
	\ref{section-lack-of-QE} and the predicates $I_G$ are true in a 
	model $K$ if and only if there exists a field extension $L/K$ such that 
	$\Gal(L/K)\cong G$. Then $\mathrm{FrobF}_{p^0}$ admits quantifier elimination.
\end{fact}
\subsection{$\FrobVF_{p^d}$}
\begin{theorem}
	\label{FrobVF-pd-QE}
	For a exponent of imperfection $d\in \omega\cup\{\infty\}$, 
	let $\FrobVF_{p^d}$ be the theory of Frobenius non-trivially valued fields 
	characteristic $p$ with exponent of imperfection
	$d$ in $\Lthetapdiv\cup\{I_G\}_{\text{$G$ finite group}}$. 
	Then $\FrobVF_{p^d}$ has quantifier elimination.
\end{theorem}
\begin{proof}
	This proof uses Fact \ref{use-embedding-to-get-QE}. 
	Suppose that $(E,V_E)$ and $(F, V_F)$ are two models of $\FrobVF_{p^d}$ 
	and $(A,V_A)$ an $\Lthetapdiv\cup\{I_G\}_G$-substructure of $(E,V_E)$ that 
	$\Lthetapdiv\cup\{I_G\}_G$-embeds into $(F,V_F)$ by $\varepsilon: A\to F$.
	We need to show that there is 
	an $\Lthetapdiv\cup\{I_G\}_G$-embedding
	of $(E,V_E)$ into $(F,V_F)$ that extends 
	$\varepsilon$ whenever $(F,V_F)$ is $|E|^+$-saturated.

	Since $(A, V_A)$ is a substructure in the 
	language $\Lthetapdiv\cup\{I_G\}_{\text{$G$ finite group}}$,
	$A$ preserves all the $0$-place predicates $I_G$, 
	thus $\Im(\Gal(E))=\Im(\Gal(F))$.

	First, notice that $A$ is a field, for the following reason. For each 
	$a\in A^\times$, 
	the maximal splitting factor of $ax-1$ is $x-a^{-1}$. Thus, 
	since $\theta_{1,0}(-1, a)=-a^{-1}\in A$, $a^{-1}\in A$. So $A$ is a field.

	Second, $E/A$ is separable, for the following reason. 
	As an $\Lp$-substructure, $A$ is closed under all the relative $p$-coordinate
	functions. If $x\in E^p$, then $\lambda_{0,0}(x)=x^{1/p}$, so if 
	furthermore $x\in A$, then $x^{1/p}\in A$. This means $A^p=A\cap E^p$. 
	Suppose that $b_1, \ldots, b_n$ are $p$-independent in $A$, then because 
	$A^p= A\cap E^p$, none of the $b_1, \ldots, b_n$ is in $E^p$. Thus there 
	exists a maximal $p$-independent subset of $\{b_1, \ldots, b_n\}$ over $E^p$, 
	which we may assume to be $\{b_1, \ldots, b_k\}$. If it 
	were true that $k< n$, Then 
	$b_{k+1}\in E^p(b_1, \ldots, b_k)$ gives a quantifier-free $\Lp$-formula 
	true in $E$ which is then also true in $A$ because $A$ is closed under 
	all the relative $p$-coordinate functions, which in turn implies that 
	$b_1, \ldots, b_n$ are $p$-dependent in $A$. Thus $k=n$ and $b_1, \ldots, 
	b_n$ remains $p$-independent in $E$. So $E/A$ is separable.  By the same 
	reasoning, $F/\varepsilon(A)$ is also separable.
	
	Third, $\varepsilon$ could be extended to an $\Lthetapdiv$-embedding of
	$(A^\alg, V_E\cap A^\alg)$. For any monic
	separable irreducible polynomial $f(X)
	\in A[X]$, if $f(X)$ has a root in $E$, then because $A$ is closed under
	the function of splitting coefficients, the maximal splitting factor of 
	$f(X)$ in $E$ is a polynomial in $A[X]$; but $f(X)$ irreducible over $A$, 
	so $f(X)$ is its own maximal splitting factor. Thus $f(X)$ splits in $E$. 
	Since $\varepsilon(A)$ is also an $\Ltheta$-substructure of $F$, 
	$f(X)$ also splits in $F$. By the same reasoning, any monic separable 
	irreducible polynomial $g(X)\in \varepsilon(A)[X]$ that has a root in $F$
	also splits both in $E$ and $F$. Therefore, by Lemma 
	\ref{extending-substructure-to-relative-separable-closure}, there 
	exists an $\Ldiv$-isomorphism 
	$$\tilde{\varepsilon}: (E\cap A^\sep, V_E\cap E\cap A^\sep) \to 
	(F\cap A^\sep, V_F\cap F\cap A^\sep),$$
	extending $\varepsilon$. Since both $E\cap A^\sep$ and $F\cap A^\sep$ 
	are relatively algebraically closed in $E$ and $F$ respectively and they
	are isomorphic as fields, their respective $\Ltheta$-substructures are 
	preserved by $\tilde{\varepsilon}$. Since $E/(E\cap A^\sep)$ and 
	$F/(F\cap A^\sep)$ remains separable, $\tilde{\varepsilon}$ is also 
	an $\Lp$-isomorphism. This means that $\tilde{\varepsilon}$ is in fact 
	an $\Lthetapdiv$-isomorphism.

	Let $(L,V_L)$ be the $\Lthetapdiv\cup\{I_G\}_G$-substructure $(E\cap A^\sep, 
	V_E\cap E\cap A^\sep)$, then $\tilde{\varepsilon}$ gives an 
	$\Lthetapdiv$-embedding of $(L,V_L)$ into $(F,V_F)$. Let 
	$M=\tilde{\varepsilon}(L)$. Then $E/L$ and $F/M$ are both regular field 
	extensions. Extend $\tilde{\varepsilon}$ to a 
	field isomorphism $\Phi_0: L^\sep \to M^\sep$. Then $\Phi_0$ induces an
	isomorphism of the absolute Galois groups $\varphi_0: \Gal(M) \to \Gal(L)$.
	By the proof of Fact \ref{Perfect-FrobF-admits-QE} in
	\cite{CherlinvddMacintyre:ThRCfields},
	 there exists an epimorphism $\varphi: \Gal(F) 
	\to \Gal(E)$ such that the following diagram commutes:
	$$
	\xymatrix{
			\Gal(E) \ar[d]_{\mathrm{res}} & \Gal(F) 
			\ar[l]_{\varphi} \ar[d]^{\mathrm{res}}
			 \\
			\Gal(L) & \Gal(M)\ar[l]^{\varphi_0},
		}
	$$

	By the Valuation Theoretic Embedding Lemma 
	\ref{valuation-theoretic-embedding-lemma}, 
	there exists an extension of $\Phi_0$ 
	to a field embedding $\Phi:E^\sep\to F^\sep$ that induces $\varphi$ with 
	$F/\Phi(E)$ regular and that $\Phi$ restricts to an $\Ldiv$-embedding 
	from $(E,V_E)$ into $(F,V_F)$.  Since $F/\Phi(E)$ is regular, 
	$\Phi$ is in fact an $\Lthetapdiv$-embedding of $(E,V_E)$ into $(F,V_E)$ 
	extending $\varepsilon$, which is then also an 
	$\Lthetapdiv\cup\{I_G\}_G$-embedding.

	These shows that $\FrobVF_{p^d}$ has quantifier elimination.
\end{proof}
\subsection{$\FrobVF_{p^d}^G$, $\PACVF_{p^d}^e$, $\PACVF_{p^d}^\omega$, 
$\SCVF_{p^d}$ and $\ACVF_p$}

From 
Theorem \ref{FrobVF-pd-QE}, we immediately 
obtain the following corollary.

\begin{corollary}
	\label{FrobVF-pd-G-QE}
	Let $G$ be a
	fixed projective profinite group with the Embedding Property, 
	$d\in \omega\cup\{\infty\}$ a fixed exponent of imperfection.
	Let $\FrobVF_{p^d}^G$
	be the class of Frobenius non-trivially valued
	fields with absolute Galois group having the same finite quotients
	as $G$, exponent of imperfection $d$, axiomatized in the 
	language $\Lthetapdiv$.  
	Then $\FrobVF_{p^d}^G$ has quantifier elimination.
\end{corollary}
\begin{proof}
	By Theorem \ref{FrobVF-pd-QE}, $\FrobVF_{p^d}^G$ has quantifier 
	elimination in the language 
	$\Lthetapdiv\cup\{I_H\}_{\text{$H$ finite group}}$. However, each 
	$0$-place predicate $I_H$ is either true in every model of $\FrobVF_{p^d}^G$,
	whence equivalent to $0=0$, 
	or false in every model of $\FrobVF_{p^d}^G$, whence equivalent to $0\neq 0$.
	Thus the conclusion follows.
\end{proof}

\begin{corollary}
	Let $e$ be a positive integer, $d$ an exponent of 
	imperfection in $\omega\cup\{\infty\}$. Then the 
	theory of $e$-free pseudo-algebraically closed non-trivially valued fields
	with exponent of imperfection $d$, axiomatized in $\Lthetapdiv$, denoted
	by $\PACVF_{p^d}^e$
	admits quantifier elimination.
\end{corollary}
\begin{proof}
	By Corollary \ref{corollary-of-Gaschutzs-lemma}, $\hat{F}_e$ has the 
	Embedding Property. It is known that a profinite finite is 
	isomorphic to $\hat{F}_e$ if and only if it has the same finite quotients as 
	$\hat{F}_e$. Thus $\PACVF_{p^d}^e$ is $\FrobVF_{p^d}^{\hat{F}_e}$, 
	which by Corollary \ref{FrobVF-pd-G-QE} admits quantifier elimination 
	in $\Lthetapdiv$.
\end{proof}

\begin{definition}
	A field $K$ is {\bf $\omega$-free} if every finite embedding problem 
	for $\Gal(K)$ is solvable, that is, given any two epimorphisms of 
	profinite groups
	$\zeta: \Gal(K)\to A$ and $\alpha: B\to A$, where $B$ is a finite group,
	there always exists an epimorphism $\gamma: \Gal(K)\to B$ such that 
	the following diagram commutes.
	$$
	\xymatrix{
		& \Gal(K) \ar[dl]_{\gamma} \ar[d]^{\zeta} \\
		B \ar[r]_{\alpha} & A
		}
	$$
\end{definition}
\begin{fact}
	[Section 3 of \cite{MR0435051}]
	Being $\omega$-free is a first order property in the language of rings $\Lr$.
\end{fact}

\begin{corollary}
	Let  $d$ be an exponent of 
	imperfection in $\omega\cup\{\infty\}$. Then the 
	theory of $\omega$-free pseudo-algebraically closed non-trivially valued fields
	with exponent of imperfection $d$, axiomatized in $\Lthetapdiv$, denoted
	by $\PACVF_{p^d}^{\omega}$
	admits quantifier elimination.
\end{corollary}
\begin{proof}
	By definition, an $\omega$-free field is Frobenius. Furthermore, in the 
	Galois formalism, for every finite group $G$, the $0$-place predicate
	$I_G$ is equivalent to $0=0$ over every model of $\PACVF_{p^d}^{\omega}$.
	The conclusion then follows from Theorem \ref{FrobVF-pd-QE}.
\end{proof}

\begin{corollary}
	The 
	theory of algebraically closed non-trivially valued fields
	with exponent of imperfection $d$, axiomatized in $\Ldiv$, denoted
	by $\ACVF_p$
	admits quantifier elimination.
\end{corollary}
\begin{proof}
	For an algebraically closed fields $K$, $\Gal(K)=\{1\}$. So $K$ is Frobenius.
	Furthermore,  for all positive integer $n$ and  $i=0, \ldots, n$,
	$\theta_{n,i}(a_0, a_1, \ldots, a_n) = a_i/a_n$ if $a_n\neq 0$.  Thus 
	a formula of the form 
	$\theta_{n,i}(x_0, x_1, \ldots, x_n) = y$ is 
	equivalent to 
	$$ x_n=0 \vee (x_n\neq 0 \wedge y x_n = x_i).$$
	For function symbols in $\Lang_p$, $\lambda_{0,0}(x)=x^{1/p}$ and 
	$\lambda_{n,i}(x; y_1, \ldots, y_n)$ is always $0$ if $n\geq1$. 
	The $0$-place predicates $I_G$ are eqiuvalent to $0=0$ or $0=1$ depending
	on whether $G$ is $1$ or not. Thus by Theorem \ref{FrobVF-pd-QE}, 
	$\ACVF_p$ has quantifier elimination. 
\end{proof}

The following is a known result. 

\begin{proposition}
	Let  $d$ be an exponent of 
	imperfection in $\omega\cup\{\infty\}$. Then the 
	theory of separably closed non-trivially valued fields
	with exponent of imperfection $d$, axiomatized in $\Lpdiv$, denoted
	by $\SCVF_{p^d}$
	admits quantifier elimination.
\end{proposition}
\begin{proof}
	The proof is similar to that of Theorem \ref{FrobVF-pd-QE}.
	We still use Fact \ref{use-embedding-to-get-QE}.
	Suppose that $(E,V_E)$ and $(F, V_F)$ are two models of $\SCVF_{p^d}$ 
	and $(A,V_A)$ an $\Lpdiv$-substructure of $(E,V_E)$ that 
	$\Lpdiv$-embeds into $(F,V_F)$ by $\varepsilon: A\to F$.
	We need to show that there is 
	an $\Lpdiv$-embedding of $(E,V_E)$ into $(F,V_F)$ that extends 
	$\varepsilon$ whenever $(F,V_F)$ is $|E|^+$-saturated.

	First the $\Lpdiv$-structure of the quotient field of $A$ is
	uniquely determined by the $\Lpdiv$-structure of $A$. Since $A$ is 
	closed under the relative $p$ coordinate function $\lambda_{0,0}(x)=x^{1/p}$,
	$A^p=A\cap E^p$. For $a\not\in A^p$, we have $\lambda_{1,1}(a;a^{p+1})=a^{-1}$,
	so $a^{-1}\in A$. Thus if $a^{-1}\not\in A$, then $a\in A^p$. 
	Thus for $b_0, b_1, \ldots, b_n\in A$ and $c_0, c_1,
	\ldots, c_n\in \Frac(A)\backslash A$,
	$$\lambda_{n,i}(b_0/c_0; b_1/c_1, \ldots, b_n/c_n)=
	\frac{\prod_{j=1}^n c_j^{i(j)}}{c_0^{1/p}}\lambda_{n,i}(b_0; b_1, 
	\ldots, b_n).$$
	Thus the $\Lp$-structure of $\Frac(A)$ is uniquely determined by 
	that of $A$.  Meanwhile, for the $\Ldiv$-structure, the value of 
	any $a/b\in \Frac(A)$ is the difference between the value of $a$ and 
	the value of $b$. Thus the $\Lpdiv$-structure of $\Frac(A)$ is 
	indeed uniquely determined by that of $A$.

	Therefore, we may assume that $A$ is a field. 

	Second, as we have seen in the proof of Theorem \ref{FrobVF-pd-QE}, 
	$E/A$ and $F/A$ are separable field extensions. So $A^\sep \cap E = 
	A^\sep$ and $A^\sep \cap F = A^\sep$. Since $A^\sep/A$ is a normal field
	extension, we can extend the $\Ldiv$-embedding $\varepsilon$ 
	to an $\Ldiv$-embedding $\tilde{\varepsilon}:A^\sep \to F$, which 
	in turn is also an $\Lpdiv$-embedding.

	Third, the Valuation Theoretic Embedding Lemma ensures that 
	$\tilde{\varepsilon}$ can be extended to an $\Lpdiv$-embedding 
	of $(E,V)$, and we are done.
\end{proof}

\section{Elementary invariants of pseudo-algebraically closed non-trivially 
valued fields}
\label{elementary-invariants}

Just like the field theoretic case (see \cite{CherlinvddMacintyre:ThRCfields}), 
the Valuation Theoretic Embedding Lemma \ref{valuation-theoretic-embedding-lemma}
enables us to characterize the elementary equivalence relation 
of pseudo-algebraically 
closed non-trivially valued fields, in terms of the so-called
coelementary equivalences. In this section, we use the notation and results from 
\cite{CherlinvddMacintyre:ThRCfields}.  See also \cite{MR1925952} for relevant 
discussions.

\begin{definition}
	For any field $K$, we use $\Abs(K)$ to denote the subfield 
	of $K$ consisting of all the numbers algebraic over the prime subfield 
	of $K$. $\Abs(K)$ is usually called the subfield of {\bf absolute numbers}
	in $K$.
\end{definition}

\begin{theorem}
	Suppose that $(K,V)$ and $(L,W)$ are two pseudo-algebraically 
	close non-trivially valued fields. Denote by $i$ and $j$ the 
	following $\Ldiv$-embeddings respectively, 
	\begin{align*}
		i: (\Abs(K), V\cap \Abs(K))\to  (K, V), \\
		j: (\Abs(L), W\cap \Abs(L))\to (L,W).
	\end{align*}
	Also denote by $\hat{i}$ and $\hat{j}$ the following induced 
	restriction maps of the absolute Galois groups respectively, 
	\begin{align*}
		\hat{i}: \Gal(K)\to \Gal(\Abs(K))\\ 
		\hat{j}: \Gal(L)\to \Gal(\Abs(L))
	\end{align*}
	Then
	$$(K,V)\equiv_{\Ldiv} (L,W)$$
	if and only if the following conditions
	hold: 
	\begin{enumerate}
		\item $K$ and $L$ have the same exponent of 
			imperfectness; 
		\item there is an $\Lr$-isomorphism 
			$$f: \Abs(L)^\sep \to \Abs(K)^\sep$$
			that restricts to an $\Ldiv$-isomorphism from
				$(\Abs(L),W\cap \Abs(L))$
			onto $(\Abs(K), V\cap \Abs(K))$ such that 
			$$(\Gal(K), \hat{i})\equiv_{\hat{f}}^\mathrm{o} 
			(\Gal(L),\hat{j}).$$
	\end{enumerate}
\end{theorem}
\begin{proof}
	The proof is more or less similar to the field theoretic case given 
	in \cite{CherlinvddMacintyre:ThRCfields}, with some care for the 
	valuations.

	Suppose that $(K,V)\equiv_{\Ldiv} (L,W)$. Then obviously 
	$K$ and $L$ have the same exponent of imperfectness. 
	Let $f^*$ be an isomorphism 
	of an elementary extension $(L^*, W^*)$ of $(L,W)$ onto an elementary 
	extension $(K^*, V^*)$ of $(K,V)$. By Fact 
	\ref{conjugate-theorem-normal-extensions}, there exists 
	an $\Ldiv$-isomorphism $f^{*\sep}: (L^*)^{\sep} \to (K^*)^{\sep}$. 
	Let $f$ be the restriction of $f^{*\sep}$ on $\Abs(L)^\sep$. Because
	$L^*$ and $K^*$ are elementary extensions of $L$ and $K$ respectively, 
	we have $\Abs(L^*)=\Abs(L)$ and $\Abs(K^*)=\Abs(K)$. Thus $f$ 
	restricts to an $\Ldiv$-isomorphism from $(\Abs(L), W\cap \Abs(L))$
	to $(\Abs(K), V\cap \Abs(K))$. Meanwhile, by the proof of Lemma 32 
	in \cite{CherlinvddMacintyre:ThRCfields}, we also have 
	$$(\Gal(K), \hat{i})\equiv_{\hat{f}}^\mathrm{o}
	(\Gal(L),\hat{j}).$$

	For the other direction, suppose that $K$ and $L$ have the same 
	exponent of imperfectness and that there is an $\Lr$-isomorphism 
	$$f: \Abs(L)^\sep \to \Abs(K)^\sep$$
	that restricts to an $\Ldiv$-isomorphism from $(\Abs(L),W\cap \Abs(L))$
	onto $(\Abs(K), V\cap \Abs(K))$ such that 
	$$(\Gal(K), \hat{i})\equiv_{\hat{f}}^\mathrm{o}
	(\Gal(L),\hat{j}),$$
	then we show that 
	$(K,V)\equiv_{\Ldiv} (L,W)$.  We also follow the proof of Proposition 33
	in \cite{CherlinvddMacintyre:ThRCfields}. We may assume that 
	both $(K,V)$ and $(L,W)$ are both $\aleph_1$-saturated. We show the 
	existence of a back-and-forth system of triples $(g,M,N)$ satisfying 
	the following Property \eqref{property-star},
	\begin{equation}
	\label{property-star}
		\left\{
		\begin{aligned}
			&\quad \text{$K/M$ and $L/N$ are regular field extensions
			where $M$ and}
			 \\
			&\quad \text{$N$ are countable subfields, $g$ 
			is an $\Lr$-isomorphism } \\
			&\quad \text{$g:M^\sep\to N^\sep$
			which restricts to an $\Ldiv$-isomorphism
			} \\
			&\quad \text{from $M$ onto $N$, and with the
			inclusions $i_M$ and $j_N$, } \\
			&\quad \text{$(\Gal(K),\hat{i}_M)
			\equiv_{\hat{g}}^\mathrm{o}
			(\Gal(L),\hat{j}_N)$ holds.}
		\end{aligned}
		\right.
	\end{equation}
	Note that first $(f, \Abs(K),\Abs(L))$ satisfies Property 
	\eqref{property-star}. 

	Now, if $(g,M,N)$ satisfies Property \eqref{property-star}, $A$ 
	a countable subset of $L$, then we show that there is a triple 
	$(g_1, M_1, N_1)$ satisfying Property \eqref{property-star} such that
	$M_1/M$ and $N_1/N(A)$ and $g_1$ extending $g$. This is guaranteed 
	by the proof of Proposition 33 of \cite{CherlinvddMacintyre:ThRCfields}
	and our Valuation Theoretic Embedding
	Lemma \ref{valuation-theoretic-embedding-lemma}. This back-and-forth
	system then proves that $(K,V)\equiv_{\Ldiv} (L,W)$.
\end{proof}

\section*{Acknowledgement}
Results in this article are from research carried out in 
Project 11626105 supported by the National Natural Science Foundation of China.

\bibliography{regenerate}{}
\bibliographystyle{apalike}

\end{document}